\documentclass[12pt]{amsart}
\usepackage{amssymb, mathtools}
\DeclarePairedDelimiter{\ceil}{\lceil}{\rceil}
\usepackage{amsmath}
\usepackage{amsfonts}
\usepackage{geometry}
\usepackage{graphicx}
\usepackage{mathrsfs,amssymb}
\usepackage{bm}
\providecommand{\U}[1]{\protect\rule{.1in}{.1in}}

\theoremstyle{plain}

\newtheorem{corollary}{Corollary}

\newtheorem{definition}{Definition}

\newtheorem{lemma}{Lemma}

\newtheorem{theorem}{Theorem}
\numberwithin{equation}{section}

\newcommand{\eps}{\epsilon}

\begin{document}
\title[Instability of the Nikodym Bounds on Manifolds]{On Instability of the Nikodym Maximal Function bounds over Riemannian Manifolds}
\author{Christopher D. Sogge}
\author{Yakun Xi}
\author{Hang Xu}

\address{Department of Mathematics\\
	Johns Hopkins University\\
	Baltimore, MD 21218, USA}
\email{sogge@jhu.edu, ykxi@math.jhu.edu,
	 hxu@math.jhu.edu}
\thanks{}
\dedicatory{ }

\begin{abstract}
We show that, for odd $d$, the $L^{\frac{d+2}2}$ bounds of Sogge \cite{sogge} and Xi \cite{xi} for the Nikodym maximal function over manifolds of constant sectional curvature, are unstable with respect to metric perturbation, in the spirit of the work of Sogge and Minicozzi \cite{MS}. A direct consequence is the instability of the bounds for the corresponding oscillatory integral operator. Furthermore, we extend our construction to show that the same phenomenon appears for any $d$-dimensional Riemannian manifold with a local totally geodesic submanifold of dimension $\ceil{\frac{d+1}2}$ if $d\ge 3$.  In contrast, Sogge's $L^\frac73$ bound for the Nikodym maximal function on 3-dimensional variably curved manifolds is stable with respect to metric perturbation. 
\end{abstract}
\maketitle
\section{Introduction}
A classical Nikodym set  $N$ in Euclidean space $\mathbb R^d$ is a measure one subset of the unit cube $[0,1]^d$, which has the property that for each $x\in N$, there is a straight line $\gamma_x$ such that $\gamma_x\cap N=\{x\}$. Because of these, the relative compliment $\mathcal N=[0,1]^d\setminus N$ must be a set of measure zero, which contains a line segment passing through each point of $N$.

It is implicit in the work of C\'ordoba that such a set $\mathcal N$ in $\mathbb R^2$ must have Hausdorff dimension 2. The Nikodym set conjecture asserts that all such set $\mathcal N$ in $\mathbb R^d$ should have Hausdorff dimension equal to $d$.

The so-called maximal Nikodym conjecture is actually a stronger conjecture that involves the following Nikodym maximal function
\[f_\delta^{**}(x)=\sup\frac{1}{|T_x^\delta|}\int_{T_x^\delta}|f(y)|dy.\]
where $T_x^\delta$ is a $1\times\delta\times\cdots\times\delta$ tube with central axis $\gamma_x$ passing through $x\in\mathbb R^d$. This maximal conjecture (formulated by C\'ordoba \cite{cordoba}) says for any $\epsilon>0$,
\begin{equation}
\label{I1}
\|f_\delta^{**}\|_{L^d(\mathbb R^{d})}\le C_\epsilon \delta^{-\epsilon}\|f\|_{L^d(\mathbb R^d)}.
\end{equation}
Interpolating with the trivial $L^1\rightarrow L^\infty$ bound, we see \eqref{I1} is equivalent to
\begin{equation}
\label{I2}
\|f_\delta^{**}\|_{L^q(\mathbb R^{d})}\le C_\epsilon\delta^{1-\frac{d}{p}-\epsilon}\|f\|_{L^p(\mathbb R^d)},
\end{equation}
where $1\le p\le d$ and $q=(d-1)p'$.

Tao \cite{tao} showed that in $\mathbb R^d$, a bound like \eqref{I2} is equivalent to the corresponding bound for the Kakeya maximal function $f^*_\delta$, thus the maximal Nikodym conjecture and the maximal Kakeya conjecture are equivalent in Euclidean space.

It is well-known (see Lemma 2.15 in \cite{bourgain}) that for a given $p$, \eqref{I2} implies that the set $\mathcal N$ must have Hausdorff dimension at least $p$.
For the case $d=p=2$, \eqref{I1} was proved by C\'ordoba \cite{cordoba}. However, it is still open for any $d\ge 3$. When $p=(d+1)/2,\ q=(d-1)p'=d+1$, \eqref{I2} follows from Drury \cite{drury} in 1983. In 1991, Bourgain \cite{bourgain} improved this result for each $d\ge3$ to some $p(d)\in((d+1)/2,(d+2)/2)$ by the so-called bush argument, where Bourgain considered the \textquotedblleft{bush}\textquotedblright{} where lots of tubes intersect at a given point. Four years later, Wolff \cite{wolff} generalized Bourgain's bush argument to the hairbrush argument, by considering tubes with lots of \textquotedblleft bushes\textquotedblright{} on them. Combining this hairbrush argument and the induction on scales, Wolff proved the following bound.
\begin{theorem}[Wolff \cite{wolff}]\label{THM1}The Nikodym maximal function satisfies
\begin{equation}
\label{I3}
\|f_\delta^{**}\|_{L^\frac{(d-1)(d+2)}{d}(\mathbb R^d)}\le C_\epsilon\delta^{1-\frac{2d}{d+2}-\epsilon}\|f\|_{L^{\frac{d+2}{2}}(\mathbb R^d)}.
\end{equation}
\end{theorem}
As mentioned before, \eqref{I3} implies that the Hausdorff dimension of $\mathcal N$ is at least $(d+2)/2$. This is still the best result for the Nikodym set conjecture when $d=3, 4$. One can get better results for larger $d$ or for the weaker Minkowski dimension, see e.g. \cite{bourgain2}, \cite{KLT}, \cite{KT}.

Even though the Kakeya set is not well-defined for curved space, one can naturally extend the definition of the Nikodym set and the corresponding maximal function to manifolds. 

\begin{definition} For a Riemannian manifold $(M,g)$. We call $\mathcal N\subset M$ a Nikodym-type set if there exists a set $\mathcal N^*\subset M$ with positive measure such that the for each point $x\in\mathcal N^*$, there exists a geodesic $\gamma_x$ passing through $x$ with $|\gamma_x\cap \mathcal N|>0$.
\end{definition}
\begin{definition} We define $f^{**}_\delta$ to be the Nikodym maximal function over a Riemannian manifold $(M,g)$, such that

 \[f_\delta^{**}(x)=\sup\frac{1}{|T_{\gamma_x}^\delta|}\int_{T_{\gamma_x}^\delta}|f(y)|dy,\]
where $T_{\gamma_x}^\delta$ is the $\delta$-neighborhood of a geodesic segment $\gamma_x$  of length $\beta$ that passes through $x$. Here $\beta$ can be chosen to be any fixed number less than one half of the injective radius of $(M,g)$.

\end{definition}

In 1997, Minicozzi and Sogge \cite{MS} showed for a general manifold, Drury's result for $p=(d+1)/2$ still holds, but surprisingly, counter examples were constructed to show that it is indeed sharp in odd dimensions. 

\begin{theorem}[Minicozzi and Sogge \cite{MS}]
\label{trivialbound}
Given $(M^d,g)$ of dimension $d\ge2$, then for $f$ supported in a compact subset $K$ of a coordinate patch and all $\epsilon>0$
\[
\|f_\delta^{**}\|_{L^q(M^d,g)}\le C_\epsilon \delta^{1-\frac{d}{p}-\epsilon}\|f\|_{L^p(M^d,g)}.
\]
where ${ p=\frac{d+1}2}$ and ${ q=(d-1)p'}$.

Furthermore, there exists an arbitrary small perturbation of the Euclidean metric, $\tilde g$, such that over $(\mathbb R^d,\tilde g)$, the above estimate breaks down for $(M^d,g)=(\mathbb R^d,\tilde g)$ if ${ p>\ceil{\frac{d+1}2}}$. \footnote{Here $\ceil{{}\cdot{}}$ is the usual  ceiling function, i.e. $\ceil{\frac{d+1}2}$ denotes the smallest integer greater than $\frac{d+1}2.$}
\end{theorem}

In 1999, Sogge \cite{sogge} managed to adapt Wolff's method for the generalized Nikodym maximal function to 3-dimensional manifolds with constant curvature. Combining a modified version of Wolff's multiplicity argument with an auxiliary maximal function, Sogge proved the following.
\begin{theorem}[Sogge, \cite{sogge}] \label{THMS}Assume that $\mathrm(M^3,g\mathrm)$ has constant sectional curvature. Then for $f$ supported in a compact subset $K$ of a coordinate patch and all $\epsilon>0$
\begin{equation}
\|f_\delta^{**}\|_{L^\frac{10}{3}(M^3,g)}\le C_\epsilon \delta^{-\frac{1}{5}-\epsilon}\|f\|_{L^\frac{5}{2}(M^3,g)}.
\end{equation}
\end{theorem}

In his proof, Sogge was able to avoid the induction on scales argument, which is hard to perform in curved space. Xi managed to generalize Sogge's result to any dimension $d\ge3$ \cite{xi}. Therefore, Wolff's bounds hold for all manifolds with constant curvature.

\begin{theorem}[Xi, \cite{xi}] \label{THMX}Assume that $\mathrm(M^d,g\mathrm)$ has constant sectional curvature with $d>3$. Then for $f$ supported in a compact subset $K$ of a coordinate patch and all $\epsilon>0$
\begin{equation}
\|f_\delta^{**}\|_{L^\frac{(d-1)(d+2)}{d}(M^d,g)}\le C_\epsilon \delta^{1-\frac{2d}{d+2}-\epsilon}\|f\|_{L^\frac{d+2}{2}(M^d,g)}.
\end{equation}
\end{theorem}

In this paper, we show that the bounds of Sogge \cite{sogge} and Xi are also unstable with respect to  metric perturbation.

\begin{theorem}\label{constant}
Given $(M^d,g)$ of dimension $d\ge3$ with constant sectional curvature, for every $\varepsilon>0$, there exists a metric $ g_\varepsilon$, such that $\|g^{ij}-g_\varepsilon^{ij}\|_{L^\infty}\le\varepsilon$, and over $(M^d, g_\varepsilon)$, the estimate 
\[
\|f_\delta^{**}\|_{L^q(M^d, g_\varepsilon)}\le C_\epsilon \delta^{1-\frac{d}{p}-\epsilon}\|f\|_{L^p(M^d, g_\varepsilon)}.
\]
breaks down if ${ p>\ceil{\frac{d+1}2}}$.
\end{theorem}

We prove this in the spirit of Minicozzi and Sogge \cite{MS}, by constructing a metric perturbation so that the Nikodym-type set could be concentrated inside a submanifold of dimension $\ceil{\frac{d+1}2}$. 

Indeed, we shall prove that this instability is generic, in the sense that for any Riemannian manifolds $(M^d,g)$ with a local totally geodesic submanifold of dimension $\ceil{\frac{d+1}2}$, we can always perturb the metric locally, to make the trivial bound as in Theorem \ref{trivialbound} to be best possible.

\begin{theorem}\label{main}
Given $(M^d,g)$ of dimension $d\ge3$ such that $M^d$ has a  local totally geodesic submanifold of dimension $\ceil{\frac{d+1}{2}}$. Then for every $\varepsilon>0$, there exists a metric $ g_\varepsilon$, such that $\|g^{ij}-g_\varepsilon^{ij}\|_\infty\le\varepsilon$, and over $(M^d, g_\varepsilon)$, the estimate
\[
\|f_\delta^{**}\|_{L^q(M^d, g_\varepsilon)}\le C_\epsilon \delta^{1-\frac{d}{p}-\epsilon}\|f\|_{L^p(M^d, g_\varepsilon)}.
\]
breaks down if ${ p>\ceil{\frac{d+1}2}}$. Indeed, there exist Nikodym type sets of dimension ${\ceil{\frac{d+1}2}}$ on $(M^d, g_\varepsilon)$.
\end{theorem}

Even though the numerology here (and in the work of Sogge and Minicozzi) may seem a bit odd at first, it can be easily understood through a simple parameter counting. For a piece of totally geodesic submanifold $\mathcal N^n\subset M^d$ of dimension $n$ to be a Nikodym-type set, there must be a collection of geodesic segments within $\mathcal N$ such that their extensions fill up a $d$-dimensional set $\mathcal N^*$. We know that the family of geodesics in $\mathcal N$ locally can be parametrized using $2n-2$ parameters, together with the 1-parameter coming from each geodesic, we must have $2n-2+1\ge d$, which clearly shows that $n=\ceil{\frac{d+1}{2}}$ is the smallest possible choice.

A direct consequence of the above results is that the corresponding oscillatory integral operator, which has the Riemannian distance function $d_g(x,y)$ as the phase function, cannot have preferable stable bound with respect to the metric perturbation. Also, in contrast, by a simple compactness argument, the $L^{\frac73}$ bound of Sogge \cite{sogge} is stable with respect to metric perturbation.

Our paper is organized as follows. In the next section, as a model case, we shall prove Theorem \ref{main} for 3-dimensional manifolds with a 2-dimensional totally geodesic submanifold, since things are much simpler in this case, yet it still provides the essential insights to this problem.  In section 3, we shall prove Theorem \ref{main} for $(2d+1)$-dimensional manifolds with the help of a simple ODE lemma. In section 4, we finish the proof of Theorem \ref{main} by pointing out how to easily generalize the proof to the even dimensional cases. Theorem \ref{constant} then follows as a corollary to Theorem \ref{main}. Finally, in the last section, we shall briefly discuss what Theorem \ref{main} tells us about the related oscillatory integral operators.

\textit{Acknowledgements.}
The third author would like to thank Professor Hamid Hezari, Prfessor Zhiqin Lu and Professor Bernard Shiffman for their constant support and mentoring. 

\section{Instability of Nikodym Bounds in Dimension $3$}
We work on a Riemannian three-fold $(M^3,g)$ with a totally geodesic, two dimensional submanifold $N^2$. In a local coordinate chart $(U_1, (x_1,x_2,x_3))$, without loss of generality, we can assume $N\cap U_1=\{x_3=0\}$ and $\frac{\partial}{\partial x_3}$ is the unit normal vector of $N\cap U_1$. Further, we assume that in coordinates $(x_1,x_2)$ on $N\cap U_1$, the cometric can be written as 
\begin{equation*}
	ds^2=dp_1^2+\tilde g^{22}(x_1,x_2)dp_2^2+dp_3^2, \quad \mbox{when } x_3=0,
\end{equation*}
where $\tilde g^{22}(x_1,x_2)=g^{22}(x_1,x_2,0)$. This can be done, for example, by choosing the polar coordinates on $N$.
Since $N\subset M$ is totally geodesic, the metric tenor must satisfy
\begin{equation*}
	\frac{\partial g^{ij}}{\partial x_3}\Big|_{x_3=0}=0, \quad \mbox{for } 1\leq i,j\leq 2.
\end{equation*}
Therefore, by taking the Taylor expansion of each $g^{ij}(x_1,x_2,x_3)$ at $x_3=0$,
\begin{equation*}
	ds^2=dp_1^2+\tilde g^{22}(x_1,x_2)dp_2^2+dp_3^2+x_3\sum_{i=1}^32h^{i3}dp_idp_3+x_3^2\sum_{1\leq i,j\leq 2}2f^{ij}dp_idp_j,
\end{equation*}
where $h^{i3}$, $1\le i\le3$, and $f^{ij}$, $1\leq i,j\leq 3$, are certain smooth functions of variable $x=(x_1,x_2,x_3)\in U_1$ for  with $f^{ij}=f^{ji}$. The factor $2$'s in the last two terms are added only to simplify our calculations.

To prove Theorem \ref{main} in this model case, we shall seek a small perturbation $g_\varepsilon$ of the metric $g$ such that in $(M,g_\varepsilon)$, there exists a Nikodym-type set of dimension $2$. For a fixed small $\varepsilon>0$, we start by fixing a function $\alpha=\alpha_\varepsilon\in C^\infty(\mathbb R)$, which satisfies that\begin{itemize}
	\item $\alpha_\varepsilon(t)=0$ for $t\leq 0$.
	\item $\alpha_\varepsilon(t)>0$ for $t>0$.
	\item $|\alpha_\varepsilon(t)|<\varepsilon$ for any $t\in \mathbb{R}$.
\end{itemize}
Let $U\subset\subset U_1$ be a relative compact subset and let $\varphi(x)\in C_0^\infty(U_1)$ be a compactly supported bump function such that $\varphi|_{U}=1$. Define 
\begin{equation*}
	g_{\varepsilon}=g^{ij}dp_idp_j+2\varphi(x)\alpha_\varepsilon(x_1)dp_2dp_3.
\end{equation*}
When $\varepsilon$ is sufficiently small, $g_\varepsilon$ is still positive definite and hence a Riemannian cometric on $M$. In the following lemma, we will investigate the geodesics in $U$ with respect to $g_\varepsilon$. For simplicity, we may assume $U=(-\delta_0,\delta_0)^3$

\begin{lemma}\label{lemma3} For $x\in U$, let
	\begin{equation*}
		H(x,p)=\frac{1}{2}\Big(p_1^2+g^{22}(x_1,x_2)p_2^2+p_3^2+x_3\sum_{i=1}^32h^{i3}p_ip_3+x_3^2\sum_{1\leq i,j\leq 2}2f^{ij}p_ip_j\Big)+\alpha(x_1)p_2p_3
	\end{equation*}
be the Hamiltonian associated to the cometric $g_\varepsilon$. 
Given $\theta\in(-1,1), a\in (-\delta_0,\delta_0)$, we denote the unique geodesic with initial position $x(0)=(0,a,0)$ and initial momentum $p(0)=(\sqrt{1-\theta^2},\theta,0)$ as $x(a,\theta;s),$
then we have $$x(a,0;s)=(s,a,0).$$
Furthermore, when $\theta=0$ and $a=0$, there exists an $0<s<\delta_0$, such that the Jacobian determinant of the map \[(a,\theta,s)\rightarrow x(a,\theta,s)\] is nonzero. 
\end{lemma}
\begin{proof}
To verify that the curves $x(a,0;s)=(s,a,0)$ are geodesics for our metric, we can look at the Hamiltonian system 
\begin{equation*}\begin{dcases} 
\dfrac{dp}{ds}=-\dfrac{\partial H}{\partial x},\\
\dfrac{dx}{ds}=\dfrac{\partial H}{\partial p}.
\end{dcases}
\end{equation*} with initial data $x(0)=(0,a,0)$, $p(0)=(1,0,0)$. This system generates the geodesic flow over the cotangent bundle, see e.g. \cite{hangzhou}. By a straightforward calculation, the Hamiltonian system becomes
\begin{equation}\label{Hamiltonian System1}
	\begin{dcases}
	\dfrac{dx_1}{ds}=p_1+x_3^2\sum_{i=1}^22f^{i1}p_i+x_3h^{13}p_3,
	\\
	\dfrac{dx_2}{ds}=\tilde g^{22}(x_1,x_2)p_2+x_3^2\sum_{i=1}^22f^{i2}p_i+x_3h^{23}p_3+\alpha(x_1)p_3,	
	\\
	\dfrac{dx_3}{ds}=p_3+x_3\sum_{i=1}^3h^{i3}p_i+x_3h^{33}p_3+\alpha(x_1)p_2,	
	\\
	\dfrac{dp_1}{ds}=-\frac{1}{2}\tilde g_{x_1}^{22}(x_1,x_2)p_2^2-x_3^2\sum_{1\le i,j\le2}f^{ij}_{x_1}p_ip_j-x_3\sum_{i=1}^3h^{i3}_{x_1}p_ip_3-{\alpha'}(x_1)p_2p_3,
	\\
	\dfrac{dp_2}{ds}=-\frac{1}{2}\tilde g_{x_2}^{22}(x_1,x_2)p_2^2-x_3^2\sum_{1\le i,j\le 2}f^{ij}_{x_2}p_ip_j-x_3\sum_{i=1}^3h^{i3}_{x_2}p_ip_3,
	\\
	\dfrac{dp_3}{ds}=-x_3\sum_{1\le i,j\le 2}2f^{ij}p_ip_j-x_3^2\sum_{1\le i,j\le2}f^{ij}_{x_3}p_ip_j-\sum_{i=1}^3h^{i3}p_ip_3-x_3\sum_{i=1}^3h^{i3}_{x_3}p_ip_3.
	\end{dcases}
\end{equation}
It is then straightforward  to check that 
\begin{equation}\label{Solution of Hamiltonian System1}
	\begin{dcases}
	x(a,0;s)=(s,a,0)\\
	p(a,0;s)=(1,0,0)
	\end{dcases}
\end{equation} solve the system \eqref{Hamiltonian System1}.

Now we verify the second part of Lemma \ref{lemma3}. Note that \[x(a,0;s)=(x_1(a,0;s),x_2(a,0;s),x_3(a,0;s))=(s,a,0).\] Thus, when $\theta=0$, the Jacobian is given by
\begin{equation*}
|J|=\left|\dfrac{\partial x(a,\theta,s)}{\partial(\theta,s,a)}\right|=\left|\begin{bmatrix}
\dfrac{\partial x_1}{\partial \theta} &\dfrac{\partial x_1}{\partial s} & \dfrac{\partial x_1}{\partial a} \\
\dfrac{\partial x_2}{\partial \theta} &\dfrac{\partial x_2}{\partial s} & \dfrac{\partial x_2}{\partial a} \\
\dfrac{\partial x_3}{\partial \theta} &\dfrac{\partial x_3}{\partial s} & \dfrac{\partial x_3}{\partial a}
\end{bmatrix}\right|=\left|\begin{bmatrix}
* &1 & 0 \\
* & 0 & 1 \\
\dfrac{\partial x_3}{\partial \theta} &0 & 0\\
\end{bmatrix}\right|
=\left|\dfrac{\partial x_3}{\partial \theta}\right|.
\end{equation*}
Now our goal is to find some $s>0$ such that $\frac{\partial x_3}{\partial \theta}\neq0$. By taking $\frac{\partial}{\partial\theta}$ on the third equation in \eqref{Hamiltonian System1} and restricting to $\theta=a=0$, with the help of \eqref{Solution of Hamiltonian System1}, we obtain
\begin{equation*}
		\dfrac{\partial}{\partial s}\dfrac{\partial x_3}{\partial \theta}=\dfrac{\partial p_3}{\partial\theta}+h^{13}(s,0,0)\dfrac{\partial x_3}{\partial\theta}+\alpha(s)\dfrac{\partial p_2}{\partial\theta}.
\end{equation*}
As $\frac{\partial p_2}{\partial\theta}$ and $\frac{\partial p_3}{\partial\theta}$ on the right hand side above are unknown, we also take $\frac{\partial}{\partial\theta}$ on the last two equations in \eqref{Hamiltonian System1} and restrict to $\theta=a=0$. Similarly, we obtain 
\begin{equation*}
	\begin{dcases}
	\dfrac{\partial}{\partial s}\dfrac{\partial p_2}{\partial\theta}=0,
	\\
	\dfrac{\partial}{\partial s}\dfrac{\partial p_3}{\partial\theta}=- 2f^{11}(s,0,0)\dfrac{\partial x_3}{\partial\theta}-h^{13}(s,0,0)\dfrac{\partial p_3}{\partial\theta}.
	\end{dcases}
\end{equation*}
The initial data $p_2(0)=\theta$ yields that $\frac{\partial p_2}{\partial\theta}=1$ for any $s$, when $\theta=a=0$. And thus $\frac{\partial x_3}{\partial\theta}(0,0;s)$ and $\frac{\partial x_3}{\partial\theta}(0,0;s)$ satisfy the following ODE system
\begin{equation}\label{ODE system}
	\begin{dcases}
	\dfrac{\partial}{\partial s}\dfrac{\partial x_3}{\partial \theta}=\dfrac{\partial p_3}{\partial\theta}+h^{13}(s,0,0)\dfrac{\partial x_3}{\partial\theta}+\alpha(s),\\
	\dfrac{\partial}{\partial s}\dfrac{\partial p_3}{\partial\theta}=- 2f^{11}(s,0,0)\dfrac{\partial x_3}{\partial\theta}-h^{13}(s,0,0)\dfrac{\partial p_3}{\partial\theta}.	
	\end{dcases}
\end{equation} 
Then we can argue that if $\frac{\partial x_3}{\partial \theta}$ is identically zero on a small interval $(0,\eps_0)$, then $\frac{\partial p_3}{\partial \theta}$ has to be identically zero by the second equation above and the fact that $\frac{\partial p_3}{\partial \theta}\big|_{s=0}=0$. However, this leads to a contradiction since $\alpha(s)\neq 0$ for $s>0.$
\end{proof}
Now we can use Lemma \ref{lemma3} to prove Theorem \ref{main} for the 3-dimensional case. Notice that in our coordinate system, $g_\varepsilon$ agrees with $g$ when $x_1\le 0$.  Moreover, since $\alpha_\varepsilon(x_1)>0$ when $x_1>0$, if we choose the point $(s,0,0)$ with $s>0$ as in Lemma \ref{lemma3},  then there is an open neighborhood $\mathcal N^*$ of the point $(s,0,0)$, such that if $x\in\mathcal N^*$, there is a unique geodesic $\gamma_x$ containing $x$ and having the property that when $x_1\le0$, $\gamma_x$ is contained in the submanifold
$N$. If we then let
\begin{equation}
f^\delta(x)=\begin{dcases}1, &\text{if}\ x\in U,\ x_1<0,\  \text{and}\ |x_3|<\delta,\\0,&\text{otherwise}, 
\end{dcases}
\end{equation}
it follows that for small fixed $x_1>0$, $(f^\delta)_\delta^{**}$ must be bounded from below by a positive constant on $\mathcal N^*$, therefore, for any $ p,\, q\ge 1$
\[\|(f^\delta)_\delta^{**}\|_{L^q(\mathcal N^*)}/\|f^\delta\|_{L^p}\ge c_0\delta^{-1/p},\]
for some $c_0>0$ depending on $\mathcal N^*$. Since
\[3/p-1<1/p\ \text{when}\ p>2,\]
we conclude that 
\[
\|f_\delta^{**}\|_{L^q(M^{3}, g_\varepsilon)}\le C_\epsilon \delta^{1-\frac{3}{p}-\epsilon}\|f\|_{L^p(M^{3}, g_\varepsilon)}
\]
cannot hold for $p>2$.

\section{Instability of Nikodym Bounds in Dimension $2d+1$}
We work on a $2d+1$ dimensional Riemannian manifold $(M^{2d+1},g)$ with a totally geodesic, $d+1$ dimensional submanifold $N^{d+1}$. In a local coordinate chart $(U_1, (x_1,x_2,\cdots,x_{2d+1}))$, without loss of generality, we may assume $N\cap U_1=\{x'=\vec 0\}$ where $x'=(x_{d+2},x_{d+3},\cdots,x_{2d+1})$, and assume $\{\frac{\partial}{\partial x_{d+2}},\frac{\partial}{\partial x_{d+3}},\cdots, \frac{\partial}{\partial x_{2d+1}}\}$ form an orthonormal basis of the normal bundle of $N\cap U_1$. Further, we assume that $(x_1,x_2,\cdots,x_{d+1})$ is the polar coordinate on $N\cap U_1$ around some point.  Thus, when $x'=0$, the cometric can be written as 
\begin{equation*}
	ds^2=dp_1^2+\sum_{2\leq i,j\leq d+1}\tilde{g}^{ij}(x_1,\cdots,x_{d+1})dp_idp_j+\sum_{d+2\leq i\leq 2d+1}dp_i^2,
\end{equation*}
where $\tilde{g}^{ij}(x_1,\cdots,x_{d+1})=g^{ij}(x_1,\cdots,x_{d+1},0,\cdots,0)$. Indeed, one may use any coordinates that gives the metric in the above form.
Since $N\subset M$ is totally geodesic, the metric tenor must satisfy
\begin{equation*}
	\frac{\partial g^{ij}}{\partial x_k}\Big|_{x'=0}=0, \quad \mbox{for } 1\leq i,j\leq d+1,\  d+2\leq k\leq 2d+1.
\end{equation*}
Therefore, by taking the Taylor expansion of each $g^{ij}(x)$ at $x'=0$,
\begin{align}\label{Cometric}
\begin{split}
		ds^2=&dp_1^2+\sum_{2\leq i,j\leq d+1}\tilde{g}^{ij}dp_idp_j+\sum_{i=d+2}^{2d+1}dp_i^2
		\\
		&+\sum_{i=1}^{2d+1}\sum_{d+2\leq k,l\leq 2d+1}2x_lh^{ikl}dp_idp_k+\sum_{1\leq i,j\leq d+1}\sum_{d+2\leq k,l\leq 2d+1}2x_kx_lf^{ijkl}dp_idp_j,
\end{split}
\end{align}
where $h^{ikl}$, $f^{ijkl}$ are certain smooth functions of variable $x\in U_1$ and $f^{ijkl}=f^{jikl}=f^{ijlk}$. 

Our goal is to find a small perturbation $g_\varepsilon$ of the metric $g$ such that in $(M,g_\varepsilon)$ there exists a Nikodym-type set of dimension $d+1$. For a fixed small $\varepsilon>0$, we start by fixing a function $\alpha=\alpha_\varepsilon\in C^\infty(\mathbb R)$, which satisfies that
\begin{itemize}
	\item $\alpha_\varepsilon(t)=0$ for $t\leq 0$,
	\item $\alpha_\varepsilon(t)>0$ for $t>0$,
	\item $|\alpha_\varepsilon(t)|<\varepsilon$ for any $t\in \mathbb{R}$.
\end{itemize}
Let $U\subset\subset U_1$ be a relative compact subset and let $\varphi(x)\in C_0^\infty(U_1)$ be a compactly supported bump function such that $\varphi|_{U}=1$. Define 
\begin{equation}\label{Perturbed Cometric}
	g_{\varepsilon}=g^{ij}dp_idp_j+2\varphi(x)\alpha_\varepsilon(x_1)\sum_{i=2}^{d+1}p_ip_{d+i}.
\end{equation}
When $\varepsilon$ is sufficiently small, $g_\varepsilon$ is still positive definite and hence a Riemannian cometric on $M$. In the following lemma, we will investigate the geodesics in $U$ with respect to $g_\varepsilon$. For simplicity, we may assume $U=(-\delta_0,\delta_0)^{2d+1}$ for some positive constant $\delta_0$. 

\begin{lemma} \label{lemmaodd}For $x\in U$, let
	\begin{align*}
		H(x,p)=&\frac{1}{2}p_1^2+\frac{1}{2}\sum_{2\leq i,j\leq d+1}\tilde{g}^{ij}p_ip_j+\frac{1}{2}\sum_{i=d+2}^{2d+1}p_i^2+\sum_{i=1}^{2d+1}\sum_{d+2\leq k,l\leq 2d+1}x_lh^{ikl}p_ip_k
		\\
		&+\sum_{1\leq i,j\leq d+1,}\sum_{d+2\leq k,l\leq 2d+1}x_kx_lf^{ijkl}p_ip_j+\alpha(x_1)\sum_{i=2}^{d+1}p_ip_{d+i}
	\end{align*}
be the Hamiltonian associated to the cometric $g_\varepsilon$. Let $\vec{0}$ be the zero row vector in $\mathbb{R}^d$.
Given $\theta=(\theta_2,\theta_3,\cdots,\theta_{d+1})\in \mathbb{R}^d$ such that $|\theta|^2=\sum_{i=2}^{d+1}\theta_i^2<1$ and $a=(a_2,a_3,\cdots,a_{d+1})\in (-\delta_0,\delta_0)^d$, we denote the unique geodesic with initial position $x(0)=(0,a,\vec{0})$ and initial momentum $p(0)=(\sqrt{1-|\theta|^2},\theta,\vec{0})$ as $x(a,\theta;s),$
then we have $$x(a,\vec{0};s)=(s,a,\vec{0}).$$
Furthermore, when $\theta=\vec{0}$ and $a=\vec{0}$, the absolute value of the Jacobian determinant of the map \[(a,\theta,s)\rightarrow x(a,\theta,s)\] is positive for $s$ in some sufficiently small interval $(0,\eps_0)$.
\end{lemma}
\begin{proof}
To verify that the curves $x(a,\vec{0};s)=(s,a,\vec{0})$ are geodesics for our metric, we can look at the Hamiltonian system 
\begin{equation*}\begin{dcases} 
\dfrac{dp}{ds}=-\dfrac{\partial H}{\partial x},\\
\dfrac{dx}{ds}=\dfrac{\partial H}{\partial p}.
\end{dcases}
\end{equation*} with initial data $x(0)=(0,\vec a,\vec{0})$, $p(0)=(1,\vec{0},\vec{0})$. This system generates the geodesic flow over the cotangent bundle. In order to avoid tedium, we will adopt the Einstein summation convention. We assume $i,j,k,l,i',j',n$ are indices within the following range, 
$1\leq i,j\leq d+1$, $d+2\leq k,l\leq 2d+1$, $2\leq i',j'\leq d+1$ and $1\leq n\leq 2d+1$. By a straightforward calculation, the Hamiltonian system becomes
\begin{equation}\label{Hamiltonian System}
	\begin{dcases}
	\dfrac{dx_1}{ds}=p_1+x_lh^{1kl}p_k
	+x_kx_lf^{1jkl}p_j,&\\
	\dfrac{dx_m}{ds}=\frac{1}{2}\tilde{g}^{mj'}p_{j'}+x_lh^{mkl}p_k
	+2x_kx_lf^{mjkl}p_j+\alpha(x_1)p_{d+m}, & 2\leq m\leq d+1,\\
	\dfrac{dx_m}{ds}=p_m+x_lh^{mkl}p_k
	+x_lh^{n ml}p_n
	+\alpha(x_1)p_{m-d},  &m\geq d+2,\\
	\dfrac{dp_1}{ds}=-\frac{1}{2}\tilde{g}^{i'j'}_{x_1}p_{i'}p_{j'}-x_lh^{n kl}_{x_1}p_{n}p_k
	-x_kx_lf^{ijkl}_{x_1}p_ip_j-{\alpha'}(x_1)p_{i'}p_{d+i'},\\
	\dfrac{dp_m}{ds}=-\frac{1}{2}\tilde{g}^{i'j'}_{x_m}p_{i'}p_{j'}-x_lh^{n kl}_{x_m}p_{n}p_k
	-x_kx_lf^{ijkl}_{x_m}p_ip_j, &2\leq m\leq d+1,\\
	\dfrac{dp_m}{ds}=-h^{n km}p_{n}p_k-x_lh^{n kl}_{x_m}p_{n}p_k
	-2x_lf^{ijml}p_ip_j-x_kx_lf^{ijkl}_{x_m}p_ip_j, &m\geq d+2.
	\end{dcases}
\end{equation}
It is not hard to see that 
\begin{equation}\label{Solution of Hamiltonian System}
	\begin{dcases}
	x(a,\vec{0};s)=(s,a,\vec{0})\\
	p(a,\vec{0};s)=(1,\vec{0},\vec{0})
	\end{dcases}
\end{equation} solve the system \eqref{Hamiltonian System}.

Now we verify the second part of Lemma \ref{lemmaodd}. Since $x(a,\vec{0};s)=(s,a,\vec{0})$, when $\theta=\vec{0}$, the Jacobian is given by
\begin{equation*}
|J|
=\left|\begin{bmatrix}
\dfrac{\partial x_1}{\partial \theta} &\dfrac{\partial x_1}{\partial s} & \dfrac{\partial x_1}{\partial a} \\
\vdots & \vdots & \vdots\\
\dfrac{\partial x_{2d+1}}{\partial \theta} &\dfrac{\partial x_{2d+1}}{\partial s} & \dfrac{\partial x_{2d+1}}{\partial a}
\end{bmatrix}\right|
=\left|\begin{bmatrix}
* &1 & \vec{0} \\
* & {\vec{0}}^\intercal & I_d \\
\left(\dfrac{\partial x_k}{\partial \theta_{j'}}\right)_{kj'} &{\vec{0}}^\intercal & \boldsymbol{0}\\
\end{bmatrix}\right|
=\left|\left(\dfrac{\partial x_k}{\partial \theta_{j'}}\right)_{kj'}\right|,
\end{equation*}
where $\boldsymbol{0}$ represents the zero $d\times d$ matrix.
Now our goal is to calculate $\left(\frac{\partial x_k}{\partial \theta_{j'}}\right)_{kj'}$. By taking the gradient $\frac{\partial}{\partial\theta}$ of the equations with $\frac{dx_{m}}{ds}$ for $d+2\leq m\leq 2d+1$ on the left hand side in \eqref{Hamiltonian System} and restricting to $\theta=a=0$, with the help of \eqref{Solution of Hamiltonian System}, we obtain
\begin{equation*}
	\dfrac{\partial}{\partial s}\dfrac{\partial x_m}{\partial \theta}=\dfrac{\partial p_m}{\partial\theta}
	+\dfrac{\partial x_l}{\partial\theta}h^{1 ml}(s,\vec{0},\vec{0})
	+\alpha(s)\dfrac{\partial p_{m-d}}{\partial\theta}, \quad \mbox{when } d+2\leq m\leq 2d+1.
\end{equation*}
As the $\frac{\partial p_m}{\partial\theta}$ in the right hand side above is unknown, we also take $\frac{\partial}{\partial\theta}$ on the last two lines of equations in \eqref{Hamiltonian System} and restrict to $\theta=a=0$. Similarly, we obtain 
\begin{equation*}
	\begin{dcases}
	\dfrac{\partial}{\partial s}\dfrac{\partial p_m}{\partial\theta}=0, & \mbox{ when } 2\leq m\leq d+1
	\\
	\dfrac{\partial}{\partial s}\dfrac{\partial p_m}{\partial\theta}=-h^{1 km}(s,\vec{0},\vec{0})\dfrac{\partial p_k}{\partial\theta}
	-2f^{11ml}(s,\vec{0},\vec{0})\dfrac{\partial x_l}{\partial\theta}, &\mbox{ when } d+2\leq m\leq 2d+1.
	\end{dcases}
\end{equation*}
For $1\leq i\leq d$, let $e_i$ be the unit vector in $\mathbb{R}^d$ whose $i$-th component is $1$. The initial data $\left(p_2(0),\cdots,p_{d+1}(0)\right)=\theta$ yields that $\frac{\partial p_m}{\partial\theta}=e_{m-1}$ for $2\leq m\leq d+1$ when $\theta=a=0$. And thus $\frac{\partial x_m}{\partial\theta}(s,\vec{0},\vec{0})$ and $\frac{\partial p_m}{\partial\theta}(s,\vec{0},\vec{0})$ for $d+2\leq m\leq 2d+1$ satisfy the following ODE system

\begin{equation}\label{ODE system2}
	\begin{dcases}
	\dfrac{\partial}{\partial s}\dfrac{\partial x_m}{\partial \theta}=\dfrac{\partial p_m}{\partial\theta}
	+h^{1 ml}(s,\vec{0},\vec{0})\dfrac{\partial x_l}{\partial\theta}
	+\alpha(s)e_{m-d-1},\\
	\dfrac{\partial}{\partial s}\dfrac{\partial p_m}{\partial\theta}=-h^{1 km}(s,\vec{0},\vec{0})\dfrac{\partial p_k}{\partial\theta}
	-2f^{11ml}(s,\vec{0},\vec{0})\dfrac{\partial x_l}{\partial\theta}.	
	\end{dcases}
\end{equation} 
We shall then need the following ODE lemma.
\begin{lemma}
For the ODE system
	\begin{equation}\label{ODE in the lemma}
	\begin{cases}
		\dot{\xi}(s)=A(s)\xi(s)+\alpha(s)\begin{pmatrix}
		I_d\\
		\boldsymbol{0}
		\end{pmatrix},\\
		\xi(0)=0,
	\end{cases}
	\end{equation}
where \begin{equation*}
	\xi(s)
	=\begin{pmatrix}
	\xi_{11}\\
	\xi_{21}
	\end{pmatrix}
\end{equation*}
is a $2d\times d$ matrix, and $A(s)$ is a fixed $2d\times 2d$ matrix with smooth entries. The solution $\xi(s)$ then satisfies that
\begin{equation*}
	(\det\xi_{11})(s)>\left(\frac{1}{2}\int_0^s\alpha(t)dt\right)^d
\end{equation*} for $s$ in some sufficiently small interval $(0,\eps_0)$. 
\end{lemma}
\begin{proof}
	Let $Z(s)$ be the fundamental matrix for the homogeneous ODE system $\dot{\xi}=A\xi$. That is to say, $Z(s)$ is a $2d\times 2d$ invertible matrix satisfies
	\begin{equation}\label{Fundamental Matrix}
	\begin{dcases}
	\dot{Z}=AZ,\\
	Z(0)=I_{2d}.
	\end{dcases}	
	\end{equation}
	Let $\eta(s)$ be a $2d\times d$ matrix-valued function such that $\xi=Z\cdot\eta$. By using this substitution, \eqref{ODE in the lemma} simplifies to 
	\begin{equation*}
		\begin{cases}
		\dot{\eta}(s)=Z^{-1}(s)\alpha(s)\begin{pmatrix}
		I_d\\
		\boldsymbol{0}
		\end{pmatrix},\\
		\eta(0)=0.
		\end{cases}
	\end{equation*}
	We now integrate both sides of the equation and obtain the solution 
	\begin{equation*}
		\eta(s)=\int_0^s\alpha(t)Z^{-1}(t)
		\begin{pmatrix}
		I_d\\
		\boldsymbol{0}
		\end{pmatrix}dt.
	\end{equation*}
	Let 
	\begin{equation*}
		\eta=\begin{pmatrix}
		\eta_{11}\\
		\eta_{21}
		\end{pmatrix},
		\quad \mbox{ and }
		Z^{-1}(t)=
		\begin{pmatrix}
		w(t) & *\\
		* & *
		\end{pmatrix},
	\end{equation*}
	where $w(t)$ is the up-left $d\times d$ matrix block.
	Thus $\eta_{11}(s)=\int_0^s\alpha(t)w(t)dt$. We denote $w(t)=\left(w_1(t),w_2(t),\cdots,w_d(t)\right)$. By a straightforward calculation,
	\begin{equation*}
		\det\eta_{11}(s)=\int_{0}^{s}\cdots\int_{0}^{s}\alpha(t_1)\alpha(t_2)\cdots\alpha(t_d)\det\left(w_1(t_1),w_2(t_2),\cdots,w_d(t_d)\right)dt_1\cdots dt_d.
	\end{equation*} 
	As a fundamental matrix, $Z$ satisfies that $Z(0)=I_{2d}$, whence $Z^{-1}(t)=I_{2d}+O(t)$ as $t\rightarrow 0$.  In particular, $w(t)=I_d+O(t)$. Recall that $\alpha(t)>0$ for any $t\in \mathbb{R}^+$. If we use $\alpha^{-1}$ to denote the anti-derivative of $\alpha$, then clearly $\det\eta_{11}(s)>\left(\frac{1}{2}\alpha^{-1}(s)\right)^d>0$ for $s$ in some sufficiently small interval $(0,\eps_0)$. 
	As $\xi=Z(s)\cdot \eta=(I_{2d}+O(s))\cdot \eta$, we obtain $\xi_{11}=\eta_{11}+O(s)$. So the result follows by possibly choosing a smaller positive $\eps_0$. 
\end{proof}
Now we may apply the above lemma to finish the proof of Lemma \ref{lemmaodd}. If we denote the $2d\times d$ matrix 
\begin{equation*}
	\xi(s)
	=\begin{pmatrix}
	\xi_{11}\\
	\xi_{21}
	\end{pmatrix}
	=\begin{pmatrix}
	\left(\dfrac{\partial x_k}{\partial \theta_{j'}}\right)_{kj'}\\
	\left(\dfrac{\partial p_k}{\partial \theta_{j'}}\right)_{kj'}
	\end{pmatrix},
\end{equation*} and $2d\times 2d$ matrix
\begin{equation*}
	A(s)=\begin{pmatrix}
	\left(h^{1kl}\right)_{kl} & I_d\\
	- \left(2f^{11kl}\right)_{kl} & -\left(h^{1lk}\right)_{kl}
	\end{pmatrix},
\end{equation*}
then \eqref{ODE in the lemma} is satisfied, we have $\det\Big(\dfrac{\partial x_k}{\partial \theta_{j'}}\Big)_{kj'}=\det(\xi_{11})>\left(\frac{1}{2}\int_0^s\alpha(t)dt\right)^d,$ completing the proof of Lemma \ref{lemmaodd}.
\end{proof}

Now we apply Lemma \ref{lemmaodd} to prove Theorem \ref{main} for odd dimensions $2d+1$. Notice that in our coordinate system, $g_\varepsilon$ agrees with $g$ when $x_1\le 0$.  Moreover, since $\alpha_\varepsilon(x_1)>0$ when $x_1>0$, if we fix a point $(s_0,\vec0,\vec0)$ with $s_0>0$, Lemma \ref{lemmaodd} implies that there is an open neighborhood $\mathcal N^*$ of the point $(s_0,\vec0,\vec0)$, such that if $x\in\mathcal N^*$, there is a unique geodesic $\gamma_x$ containing $x$ and having the property that when $x_1\le0$, $\gamma_x$ is contained in the submanifold
$\{x:x'=\vec 0\}$. If we then let
\begin{equation}
f^\delta(x)=\begin{dcases}1, &x\in U,\ x_1<0,\  \text{and}\ |x'|=|(x_{d+2},\ldots,x_{2d+1})|<\delta,\\0,&\text{otherwise}, 
\end{dcases}
\end{equation}
it follows that for small fixed $x_1>0$, $(f^\delta)_\delta^{**}$ must be bounded from below by a positive constant on $\mathcal N^*$, therefore, for any $q, p\ge 1$
\[\|(f^\delta)_\delta^{**}\|_{L^q(\mathcal N^*)}/\|f^\delta\|_{L^p}\ge c_0\delta^{-d/p},\]
for some $c_0>0$ depending on $\mathcal N^*$. Since
\[(2d+1)/p-1<d/p\ \text{when}\ p>d+1,\]
we conclude that 
\[
\|f_\delta^{**}\|_{L^q(M^{2d+1}, g_\varepsilon)}\le C_\epsilon \delta^{1-\frac{2d+1}{p}-\epsilon}\|f\|_{L^p(M^{2d+1}, g_\varepsilon)}
\]
cannot hold for $p>d+1$.

\section{Instability of Nikodym Bounds in Dimension $2d$}
In this section, we work on a $2d$ dimensional Riemannian manifold $(M^{2d},g)$ with a totally geodesic, $d+1$ dimensional submanifold $N^{d+1}$. As the construction is similar as in the odd dimensional case, we will only indicate the differences here. By using the same simplifications as before, like in \eqref{Cometric}, we can write the Riemannian cometric as  
\begin{align*}
ds^2=&dp_1^2+\sum_{2\leq i,j\leq d+1}\tilde{g}^{ij}dp_idp_j+\sum_{i=d+2}^{2d}dp_i^2
\\
&+\sum_{i=1}^{2d}\sum_{d+2\leq k,l\leq 2d}2x_lh^{ikl}dp_idp_k+\sum_{1\leq i,j\leq d+1}\sum_{d+2\leq k,l\leq 2d}2x_kx_lf^{ijkl}dp_idp_j,
\end{align*}
where $h^{ikl}$, $f^{ijkl}$ are certain smooth functions of variable $x\in U_1$ and $f^{ijkl}=f^{jikl}=f^{ijlk}$. In this case, instead of \eqref{Perturbed Cometric}, we take 
\begin{align*}
g_\varepsilon=g^{ij}dp_idp_j+2\alpha_\varepsilon(x_1)\sum_{i=3}^{d+1} dp_idp_{d-1+i}
\end{align*}
to be the perturbed cometric. Therefore, by the proof of Lemma \ref{lemmaodd}, given $\theta=(\theta_2,\theta_3,\cdots,\theta_{d}) \in \mathbb{R}^{d-1}$ with $|\theta|^2=\sum_{i=2}^{d}\theta_i^2<1$ and $a=(a_2,a_3,\cdots,a_{d+1})\in (-\delta_0,\delta_0)^{d}$, we denote the unique geodesic with initial position $x(0)=(0,a,\vec{0})$ and initial momentum $p(0)=(\sqrt{1-|\theta|^2},\theta,\vec{0})$ as $x(a,\theta;s)$,
then we have $$x(a,\vec{0};s)=(s,a,\vec{0}).$$ Furthermore, the absolute value of the Jacobian determinant of the map \[(a,\theta,s)\rightarrow x(a,\theta,s)\] is positive for $\theta=\vec{0}$, $a=\vec{0}$ and $s$ in some sufficiently small interval $(0,\delta)$. Consequently, if we fix a point $(s_0,\vec0,\vec0)$ with $s_0>0$, there is an open neighborhood $\mathcal N^*$ of the point $(s_0,\vec0,\vec0)$, such that if $x\in\mathcal N^*$, there is a unique geodesic $\gamma_x$ through $x$ and lying in the submanifold
$N\cap U=\{x:x'=(x_{d+2},x_{d+3},\cdots,x_{2d})=\vec 0\}$ when $x_1\le0$. As a result, if we put
\[f(x)=
\begin{cases}
1, &\text{if}\ x\in U,\ x_1<0 \  \text{and}\ |x'|=|(x_{d+2},\ldots,x_{2d})|<\delta, \\
0, &\text{otherwise},
\end{cases} \]
it follows that for small fixed $x_1>0$, $f_\delta^{**}$ must be bounded from below by a positive constant on $\mathcal N^*$, therefore, for any $q, p\ge 1$
\[\|f_\delta^{**}\|_{L^q(\mathcal N^*)}/\|f\|_{L^p}\ge c_0\delta^{-(d-1)/p},\]
for some $c_0>0$ depending on $\mathcal N^*$. Since
\[(2d)/p-1<(d-1)/p\ \text{when}\ p>d+1,\]
we conclude that 
\[
\|f_\delta^{**}\|_{L^q(M^{2d}, g_\varepsilon)}\le C_\epsilon \delta^{1-\frac{2d}{p}-\epsilon}\|f\|_{L^p(M^{2d}, g_\varepsilon)}
\]
cannot hold for $p>d+1$.

\section{Instability for oscillatory integral bounds}
Following the same strategy as on \cite[p. 290]{fio}, one may easily derive the following instability results for the related oscillatory integrals.

We consider the oscillatory integral operator
\begin{equation}
S^g_\lambda f(x)=\int_{M}e^{i{\lambda d_g(x,y)}}a(x,y)f(y)\,dy,
\end{equation}
where $a(x,y)\in C_0^\infty$ vanishes near the diagonal. Then we have
\begin{corollary}

Given $(M^d,g)$ of dimension $d\ge3$ such that $M^d$ has a  local totally geodesic submanifold of dimension $\ceil{\frac{d+1}{2}}$. Then for every $\varepsilon>0$, there exists a metric $ g_\varepsilon$, such that $\|g^{ij}-g_\varepsilon^{ij}\|_\infty\le\varepsilon$, and over $(M^d, g_\varepsilon)$, the estimate
\[
\|S_\lambda^{g_\varepsilon} f\|_{L^q(M^d, g_\varepsilon)}\le C_\epsilon \lambda^{-\frac{d}{p}+\epsilon}\|f\|_{L^p(M^d, g_\varepsilon)}.
\]
breaks down if \begin{equation}p>\begin{dcases}\frac{2(3d+1)}{3(d-1)}, &\text{when }d\ge3\text{ is odd},\\
\frac{2(3d+2)}{3d-2}, &\text{when }d\ge4\text{ is even}.
\end{dcases}\end{equation}

\end{corollary}

\bibliography{Nikodym}

\bibliographystyle{plain}
\end{document}